\documentclass[a4paper,12pt]{amsart}          
\usepackage{amsfonts,amsmath,latexsym,amssymb} 
\usepackage{amsthm}      
\usepackage[dvipsnames]{xcolor}          

\usepackage{float}
\usepackage{tikz}
\usepackage{circuitikz}
\usepackage{hyperref}
\usepackage{mathtools}
\usepackage{enumitem}

\allowdisplaybreaks

\newtheorem{theorem}{Theorem}           
\newtheorem{lemma}{Lemma}               
\newtheorem{corollary}{Corollary}

\theoremstyle{definition}

\newtheorem{definition}{Definition}
\newtheorem{example}{Example}

\newcommand{\h}{\overline{h}}

\newcommand{\K}{\mathbb{K}}
\newcommand{\R}{\mathcal{R}}
\renewcommand{\L}{\mathcal{L}}
\newcommand{\F}{\mathfrak{F}}

\author{Anna Muranova}
 \address{Anna Muranova: Faculty of Mathematics and Computer Science,
University of Warmia and Mazury,
ul. Sloneczna 54, 10-710 Olsztyn, Poland} 

 \email{anna.muranova@matman.uwm.edu.pl}

\thanks{
\textit{Keywords}: weighted graphs, Laplace operator, non-Archimedean field, ordered field, Levi-Civita field, dual Cheeger constant.\\
\textit{Mathematics Subject Classification 2010: }{05C50,
05C22, 47A75, 12J15, 39A12}}

    \title[Dual Cheeger constant for OF-graphs]{Dual Cheeger constant for weighted graphs over ordered fields}

\begin{document}

\maketitle

\medskip
    \begin{abstract}

We consider a dual Cheeger constant $\h$ for finite graphs with edge weights from an arbitrary real-closed ordered field.  We obtain estimates of $\h$ in terms of number of vertices in graph. Further, we estimate the largest eigenvalue for the discrete Laplace operator in terms of $\h$ and show the sharpness of estimates. As an example we consider graphs over non-Archimedean field of the Levi-Civita numbers.
    
    \end{abstract}

\section*{Introduction}
The Cheeger isoperimetric constant $h$ was firstly introduced for compact Riemannian manifold by Jeff Cheeger \cite{Cheeger} in 1969 and then considered for graphs (e.g. \cite{Chung, Mohar}). The dual Cheeger constant $\overline h$ was introduced directly on weighted graphs in \cite{BauerJost}. The dual Cheeger constant shows how close the graph is to bipartite and it is used in estimation of the largest eigenvalue of discrete Laplacian.  Classical approach to weighted graph assume, that edge weights are positive real numbers. In this paper we consider the model, which was introduced in previous works by author (see \cite{Muranova1, Muranova2, MuranovaPreprint22}), where edge weights are positive elements of an arbitrary  real-closed ordered field $(\K, \succ)$. We introduce the notion of dual Cheeger constant $\overline h$ for finite graphs analogously to the classical case of field $\Bbb R$ and prove its estimates in terms of number $N$ of vertices in graph. We show, that for even $N$ holds
$$
\overline h\succeq \dfrac{N}{2(N-1)}
$$
and for odd $N$ holds
$$
\overline h\succeq \dfrac{N+1}{2N}.
$$
These estimates improve the result in \cite{BauerJost}, which states that $\overline h\ge \dfrac{1}{2}$ for the field $\Bbb R$. The sharpness of our result follows from examples in \cite{BauerJost}, since the equality holds for a complete graph with all the weights equal $1$ over $\mathbb R$(and over any real-closed field, due to the same calculations). 

Further, we show that for any $\mathbb K$ the largest eigenvalue of the Laplace operator can be estimated as 
\begin{equation}\label{eqLambdah}
2 \overline h\preceq \lambda_{N-1} \preceq 1+\sqrt{1-(1-\overline h)^2},
\end{equation}
where the both equalities hold for any bipartite graph.

In Section \ref{LC}, we consider graphs over a non-Archimedean field of the Levi-Civita numbers $\R$ and show, that left inequality in \label{eqLambdah} can hold also for complete graphs, while the right one is precise for them in a sense, that there is a sequence of graphs, for which the difference $|\lambda_{N-1} - 1+\sqrt{1-(1-\overline h)^2}|$ converges in an order topology. 

\section{Preliminaries}

\subsection{Real-closed ordered field}
Firstly we remind a concept of real-closed ordered field. More detailed description can be found in textbooks on Algebra see e. g. \cite{Lang, Waerden}.

\begin{definition}
A field $\K$ is called {\em ordered}, if there exist a subset $\K^+~\subset \K$ with the following properties:
\begin{enumerate}
\item 
$0\not \in \K^+$,
\item
for any non-zero element $k\in \K$ either $k\in \K^+$ or $-k\in \K^+$,
\item
if $k_1, k_2\in \K^+$, then $k_1+k_2\in \K^+$ and $k_1k_2\in \K^+$.

\end{enumerate}
\end{definition}
The total order in the  ordered field is then defined as follows: for $k_1,k_2 \in \K$ we write $k_1\succ k_2$ (or, equivalently, $k_2\prec k_1$) if $(k_1-k_2)\in \K^+$. Further $k_1\succeq k_2$ (equiv. $k_1\preceq k_2$) if $(k_1-k_2)\in \K^+\cup\{0\}$. 

An ordered field $\K$ is {\em non-Archimedean} if there exist an {\em infinitesimal}, i.e. $\epsilon \in \K$ such that 
$$
{\epsilon}\prec \dfrac{1}{n}=\dfrac{1}{\underbrace{1+\dots+1}_{n\mbox{ times}}}
$$
for any $n\in \Bbb N\subset\K$. Otherwise, the field is called {\em Archimedean}. All Archimedean fields are isomorphic to some (not necessary proper) subfield of real numbers. Examples of non-Archimedean fields are ordered field of rational functions \cite[p. A.VI.21]{Bourbaki}, Levi-Civita field \cite{Berz, Hall}, super-real field \cite{DalesHughWoodin}.

By \emph{absolute value of $k\in \K$} we mean $|k|\in \K^+\cup\{0\}$ defined as 
\begin{equation*}
|k|=\begin{cases}
k, \mbox{ if }k\succeq 0\\
-k, \mbox{otherwise.}
\end{cases}
\end{equation*}

The convergence of sequence $\{k_n\}_{n\in \Bbb N}$ to $k$ in ordered topology of the ordered field means by definition, that for any element $a\in \K$ exists $N_0\in \Bbb N$ such that for any $n\ge N_0$ we have
$$
|k_n-k|\preceq a.
$$

There are several equivalent definitions of a real-closed field \cite[p. 38]{DalesHughWoodin}. We remind some of them.\\

An ordered field $\K$ is {\em real-closed} if one of the following holds:
\begin{enumerate}
\item
$\K$ has no proper algebraic extension to an ordered field,
\item
the complexification of $\K$ is algebraically closed,
\item
every positive element in $\K$ has a square root and every polynomial over $\K$ of odd degree has a root in $\K$.
\end{enumerate}

It is well known, that order can be introduced in any real-closed field. Moreover, any ordered field has an algebraic real-closed extension closure, which is unique up to isomorphism, preserving the order (Artin-Schreier Theorem, see \cite[p. 38]{DalesHughWoodin}). Therefore, further in this note we consider exclusively real-closed ordered fields, e.g. algebraic numbers, real numbers, Levi-Civita field.

\subsection{Graphs over ordered field}
The concept of graph over an ordered field was introduced in previous works by author (see \cite{Muranova1, Muranova2}). Graphs over the field $\Bbb R$ of real numbers are classical weighted graphs. Here we shortly remind the definition.  field. 
\begin{definition}
Let $\K$ be a real-closed ordered. \emph{A graph oven an ordered field $K$ (OF-graph)}  is a couple $(V, b)$, where $V$ is a set of vertices (i.e. arbitrary set)
and $b:V\times V\rightarrow \K^+\cup\{0\}$ satisfies the following properties:
\begin{enumerate}
\item
$b(x,y)=b(y,x)$ for any $x,y\in V$,
\item
$b(x,x)= 0$ for any $x\in V$.
\end{enumerate}
\end{definition}
The last property mean that we consider graphs {\em without loops}. We write $x\sim y$ whenever $b(x,y)\ne 0$ and say that there is an {\em edge} between $x$ and $y$. 

A {\em subgraph} of a graph $(V,b)$ is any graph $(U,b\left|_{U\times U}\right.)$, where $U\subset V$.

A {\em path} between any vertices $x,y\in V$ is any sequence $\{x_i\}_{i=1}^n, n\in \Bbb N$ such that
$$
x=x_0\sim x_1\sim x_2\sim \dots \sim x_n=y.
$$
A graph is called {\em connected}, if there exists a path between any two vertices of it.

A {\em connected component} of a graph is any its connected subgraph that is not part of any larger connected subgraph. Therefore, set of vertices of graph is a disjoint union of vertices of all its connected components.

Further in this note we consider finite ($\#V<\infty$) graphs. For the sake of brevity futher we refer to a graph by its set of vertices (e.g. $V$).

A graph is called {\em bipartite}, if there exist a partition $V=V_1\cup V_2$ of vertices,
such that from $x\sim y$ follows that $x\in V_1, y\in V_2$ or $x\in V_2, y\in V_1$.

A graph is called {\em complete}, if $x\sim y$ for any two $x,y\in V$ with $x\ne y$.

The \emph{normalized} weight on edges is defined by
\begin{equation}
p(x,y)=\dfrac{b(x,y)}{b(x)}, \mbox{ for any }x,y \in V,
\end{equation}
where $b(x)=\sum_y b(x,y)$. If $b(x)=0$ then the vertex $x$ is called {\em {isolated vertex}}. Further we consider graphs without isolated vertices, therefore $p(x,y)$ is well defined.

For any $V_1\subset V$ we denote
$$
b(V_1)=\sum_{x\in V_1}b(x)$$
and for any $V_1,V_2\subset V$ with $V_1\cap ~V_2=\emptyset$ we denote
$$
b(V_1,V_2)=\sum_{\substack{x\in V_1,\\y\in V_2}}b(x,y).
$$

\subsection{Laplace operator}
Our main interest is {normalized Laplace operator (Laplacian)} defined on the set of functions
$$
\mathfrak F_V=\{f \mid f:V\to \K\}
$$
by 
\begin{equation}
\mathcal L f(x)=\sum_{y\in V} (f(x)-f(y))\dfrac{b(x,y)}{b(x)}=\sum_{y\in V} (f(x)-f(y))p(x,y).
\end{equation}
An {\em eigenvalue} of $\mathcal L$ is defined as $\lambda$ such that $\mathcal Lf=\lambda v$, where $v$ is called {\em eigenfunction}.
It is proven in \cite{MuranovaPreprint22} that all $N=\#V$ eigenvalues of the Laplacian for a connected graph belong to the same field $\K$ as edge weights and all the eigenfunctions are in $\mathfrak F_V$. Moreover, it is shown there that the largest eigenvalue of the connected graph
satisfies
 $$
\dfrac{N}{N-1}\preceq\lambda_{N-1}\preceq 2
$$ and $\lambda_{N-1}=2$ if and only if the graph is bipartite.

For non-connected graph the eigenvalues are union of eigenvalues of its connected components (with multiplicities) and eigenfunctions are extension of eigenfunctions of connected components to the whole graph, which are $0$ on the other components.

Here and further we order eigenvalues of Laplacian as 
$$
0=\lambda_0\preceq \lambda_1\preceq\lambda_2\preceq \dots \preceq \lambda_{N-1}\preceq 2.
$$
 The multiplicity of the eigenvalue $0$ is equal to the number of connected components of the graph.

For any function $f\in \mathfrak F_V$ we denote  
$$
V_f^+=\{x\in V\mid f(x)\succ 0\}
$$ 
and
$$
V_f^-=\{x\in V\mid f(x)\prec 0\}.
$$
We introduce an analogue of scalar product on $\F_V$ by 
$$
\langle f,g\rangle=\sum_{x\in V}f(x)g(x)b(x).
$$
for any $f,g \in \F_V$.
The Green formula (see e.g. \cite{Muranova1}) states, that  
$$
\langle \L f,g\rangle=\dfrac{1}{2}\sum_{x,y\in V}(f(x)-f(y))(g(x)-g(y))b(x,y).
$$

\section{Dual Cheeger constant}
The dual Cheeger constant for weighted graphs was introduced in \cite{BauerJost}. Here we introduce it in the same way for OF-graphs.
\begin{definition}
Let $V$ be an OF-graph with $N=\#V$ vertices. Then its {\em dual Cheeger constant} is defined by
\begin{equation}\label{dualCheeger}
\overline h=\max_{\substack{\emptyset \ne V_1\subset V,\\ \emptyset \ne V_2\subset V,\\V_1\cap V_2=\emptyset}}\dfrac{2b(V_1,V_2)}{b(V_1)+b(V_2)},
\end{equation}
i.e. the maximum is taken over all pairs of non-empty disjoint subsets of $V$.
\end{definition}

\begin{lemma}
For any connected graph the dual Cheeger constant satisfies
\begin{equation}\label{hpreceq}
\overline h\preceq 1.
\end{equation}
Moreover, $\overline h=1$ if and only if the graph is bipartite.
\end{lemma}
The proof follows the same outline as in \cite{BauerJost}.
\begin{proof}
 The inequality \eqref{hpreceq} follows immediately, since $b(V_1,V_2)\preceq b(V_l)$ for $l=1,2.$

Further, let $V$ is bipartite with the partition $V=V_1\cup V_2$. Then equality in \eqref{hpreceq} is attained on this partition. 
Let now $\overline h=1$. Then 
$$
2b(V_1,V_2)=b(V_1)+b(V_2),
$$
i.e. there is no edges between the vertices from the same $V_l$, for $l=1,2$. Moreover, since the graph is connected, $V_3=\emptyset$. Therefore, the graph is bipartite and $V=V_1\cup V_2$ is its partition.
\end{proof}

\begin{theorem}\label{MainThm1}

\textnormal {(a)} For a graph $V$ with $N=2K\ge 2$ vertices and at least one edge, we have
$$
\overline h \succeq \frac{N}{2(N-1)}.
$$
\textnormal {(b)}  For a graph $V$ with $N=2K+1>2$ vertices and at least one edge, we have
$$
\overline h \succeq \dfrac{N+1}{2N}.
$$
\end{theorem}

\begin{proof}
(a) Let us consider only such $V_1, V_2\subset V$  that $V=V_1\cup V_2$ is a partition of the graph with  $\#V_1=\# V_2=K$. There are ${N\choose K}=\frac{N!}{K!(N-K)!}$ such partitions.  For each such $V_1, V_2$ we define
$$
\overline h (V_1,V_2):=\frac{2b(V_1,V_2)}{b(V_1)+b(V_2)}.
$$
therefore $\overline h\succeq \overline h(V_1,V_2)$.

For every such a partition $V=V_1\cup V_2$ we have 
$$
\overline h(V_1,V_2)=\frac{2\sum_{x\in V_1,y\in V_2}b(x,y)}{b(V)}=\frac{2\sum_{x\in V_1,y\in V_2}b(x,y)}{\sum_{x,y\in V}b(x,y)},
$$
and thus
\begin{equation}\label{thm1ineq}
\frac{2\sum_{x\in V_1,y\in V_2}b(x,y)}{\sum_{x,y\in V}b(x,y)}\preceq \overline h
\end{equation}
Note that $b({x_0},y_0)$ for some $x_0,y_0\in V$ is in numerator if and only if ${x_0}\in V_1$ and ${y_0}\in V_2$. It's easy to count that among all the considered partitions
there are ${{N-2}\choose {K-1}}$ such partitions for any given ordered pair $(x_0,y_0)$ . 
Now we sum up all inequalities \eqref{thm1ineq}, rewrite the numerator using the last observation and obtain
$$
{{N-2}\choose {K-1}}\frac{2\sum_{x_0,y_0\in V}b(x_0,y_0)}{\sum_{x,y\in V}b(x,y)}\preceq\overline h {N\choose K}
$$
from where follows
$$
2{{N-2}\choose {K-1}}\preceq\overline h {N\choose K}
$$
and, further,
$$
\overline h \succeq \dfrac{2(N-K)K}{N(N-1)}=\dfrac{K}{(N-1)}=\dfrac{N}{2(N-1)},
$$
since $N=2K$.


(b) The idea of the proof is the same, as in the part (a), but the corresponding partitions are $\#V_1=K, \# V_2=K+1$. Then, following the same outline as in (a), we get
$$
\overline h \succeq \dfrac{2(N-K)K}{N(N-1)}=\dfrac{N-K}{N}=\dfrac{N+1}{2N},
$$
since $N=2K+1$.

\end{proof}

Therefore, the following Corollary, which improves the result $\overline h\ge\frac{1}{2}$ in \cite{BauerJost}, holds:
\begin{corollary}
For any finite graph with at least one edge
$$
\overline h \succ \dfrac{1}{2}.
$$
\end{corollary}

\section{Largest eigenvalue of the graph Laplacian}
In this section we show the relation between dual Cheeger constant and the largest eigenvalue $\lambda_{N-1}$ of the discrete Laplacian for the finite graph over an ordered field. For the field $\Bbb R$ it is done in \cite{BauerJost} and the result is similar to the result, which relate Cheeger constant and smallest non-zero eigenvalue (see e. g. \cite{Chung}, \cite{Grigoryan} for $\Bbb R$ and \cite{MuranovaPreprint22} for ordered fields).

\begin{theorem}\label{thmlambda}
The largest eigenvalue $\lambda_{N-1}$ of the graph Laplacian satisfies
$$
2 \overline h\preceq \lambda_{N-1} \preceq 1+\sqrt{1-(1-\overline h)^2}.
$$
\end{theorem}

For the proof of this theorem we need the following important lemma, which is proved in \cite{BauerJost} for $\Bbb K=\Bbb R$ , but there the proof uses integrals.
Here we present the proof of the lemma, which is analogous to the prove of Theorem 2 in \cite{MuranovaPreprint22} and uses the technique, used in \cite{Chung} for the proof of the relation between the smallest non-zero eigenvalue of Laplacian and Cheeger constant.

\begin{lemma}
Let $(V,b)$ be a graph without isolated vertices and $f\in \F_V$ be such that $V^+_f\ne\emptyset$ and $V\setminus V^+_f\ne\emptyset$. Let
\begin{equation*}
h(f)=\min_{\emptyset\ne S\subset V^+_f}\dfrac{b(S, V\setminus S)}{b(S)}
\end{equation*}
and 
$$
f_+(x)=\begin{cases}
f(x), x\in V^+_f,\\
0, \mbox{ otherwise},
\end{cases}
$$
then
\begin{equation*}
1-\sqrt{1-h^2(f)}\preceq \dfrac{\langle\L f_+, f_+\rangle}{\langle f_+, f_+\rangle} \preceq 1+\sqrt{1-h^2(f)}.
\end{equation*}

\end{lemma}

\begin{proof}
Let us denote
\begin{equation*}
W:=\dfrac{\langle\L f_+, f_+\rangle}{\langle f_+, f_+\rangle}.
\end{equation*}
By Green formula we can rewrite $W$ as 
\begin{equation*}
W=\dfrac{\frac{1}{2}\sum_{x,y\in V}(f_+(x)-f_+(y))^2b(x,y)}{\sum_{x\in V}f_+^2(x)b(x)}.
\end{equation*}
Then multiplying the numerator and the denominator by $\sum_{x,y\in V}(f_+(x)+f_+(y))^2b(x,y)$ and using Cauchy-Schwarz inequality we get
\begin{equation}\label{W1}
W\succeq \dfrac{\frac{1}{2}\left(\sum_{x,y\in V}b(x,y)\left|f_+^2(x)-f_+^2(y)\right|\right)^2}{\left(\sum_{x\in V}f_+^2(x)b(x)\right)\left(\sum_{x,y\in V}(f_+(x)+f_+(y))^2b(x,y)\right)}.
\end{equation}
From the other hand,
\begin{align*}
&W\cdot {\sum_{x\in V}f_+^2(x)b(x)}=\langle\L f_+,f_+\rangle\\
&=2\sum_{x\in V}f_+^2(x)b(x)-\left(\sum_{x\in V}f_+^2(x)b(x)+\sum_{x,y\in V}f_+(x)f_+(y)b(x,y)\right)\\
&=2\sum_{x\in V}f_+^2(x)b(x)-\left(\sum_{x,y\in V}f_+(x)(f_+(x)+f_+(y))b(x,y)\right)
\\
&=2\sum_{x\in V}f_+^2(x)b(x)-\left(\sum_{x,y\in V}f_+(y)(f_+(x)+f_+(y))b(x,y)\right)\\
&=2\sum_{x\in V}f_+^2(x)b(x)-\frac12\left(\sum_{x,y\in V}(f_+(x)+f_+(y))^2b(x,y)\right),
\end{align*}
where in the last two lines we have switched notations for $x,y$ in the second sum and sum up two lines and divided by two in the last line. Using this in denominator of \eqref{W1} we get
\begin{equation}\label{W2}
W\succeq\dfrac{\frac{1}{2}\left(\sum_{x,y\in V}b(x,y)\left|f_+^2(x)-f_+^2(y)\right|\right)^2}{\left(\sum_{x\in V}f_+^2(x)b(x)\right)^2\left(4-2 W\right)}.
\end{equation}
Now our goal is to estimate the numerator of \eqref{W2}
$$
A:=\frac{1}{2}\left(\sum_{x,y\in V}b(x,y)\left|f_+^2(x)-f_+^2(y)\right|\right)^2
$$ 
in terms of $h(f)$.
Firstly, let us order all the vertices of the graph as $x_1,x_2\dots,x_N$ in the way that $f(x_i)\preceq f(x_{i+1})$. Further, let us denote 
$$
S_k=\{x_{k+1},\dots,x_N\}\subset~V\;\;\;\mbox{ and }\;\;\;n_0=\max_i\{i\mid x_i\not \in V^+_f, 0\le i<N \},
$$
i.e. for any $k\ge n_0$ we have $S_k\subset~V^+_f$.

Now we can rewrite $A$ as
\begin{align*}
A&=2\left(\sum_{i=1}^{N-1}\sum_{j=i+1}^N b(x_i,x_j)\left(f_+^2(x_j)-f_+^2(x_i)\right)\right)^2\\
&=2\left(\sum_{i=1}^{N-1}\sum_{j=i+1}^N b(x_i,x_j)\sum_{k=i}^{j-1}\left(f_+^2(x_{k+1})-f_+^2(x_k)\right)\right)^2\\
&=2\left(\sum_{i=1}^{N-1}\sum_{j=i+1}^N \sum_{k=i}^{j-1}b(x_i,x_j)\left(f_+^2(x_{k+1})-f_+^2(x_k)\right)\right)^2\\
&=[\mbox{change limits of summation for $k$ and $j$}]\\
&=2\left(\sum_{i=1}^{N-1}\sum_{k=i}^{N-1} \sum_{j=k+1}^{N}b(x_i,x_j)\left(f_+^2(x_{k+1})-f_+^2(x_k)\right)\right)^2\\
&=[\mbox{change limits of summation for $i$ and $k$}]\\
&=2\left(\sum_{k=1}^{N-1}\underbrace{\sum_{i=1}^{k} \sum_{j=k+1}^{N}b(x_i,x_j)}_{b(S_k,V\setminus S_k)}\left(f_+^2(x_{k+1})-f_+^2(x_k)\right)\right)^2\\
&=2\left(\sum_{k=1}^{N-1}{b(S_k,V\setminus S_k)}\left(f_+^2(x_{k+1})-f_+^2(x_k)\right)\right)^2\\
&=2\left(\sum_{k=n_0}^{N-1}{b(S_k,V\setminus S_k)}\left(f_+^2(x_{k+1})-f_+^2(x_k)\right)\right)^2\\
&\succeq 2\left(\sum_{k=n_0}^{N-1}{h(f)b(S_k)}\left(f_+^2(x_{k+1})-f_+^2(x_k)\right)\right)^2\\
&= 2h^2(f)\left(\sum_{k=n_0}^{N-1}{\sum_{i=k+1}^{N}b(x_i)}\left(f_+^2(x_{k+1})-f_+^2(x_k)\right)\right)^2\\
&=[\mbox{change limits of summation for $i$ and $k$}]\\
&=2h^2(f)\left(\sum_{i=n_0+1}^{N}{\sum_{k=n_0}^{i-1}b(x_i)}\left(f_+^2(x_{k+1})-f_+^2(x_k)\right)\right)^2\\
&=2h^2(f)\left(\sum_{i=n_0+1}^{N}{b(x_i)}\left(f_+^2(x_i)-f_+^2(x_{n_0})\right)\right)^2\\
&=2h^2(f)\left(\sum_{i=n_0+1}^{N}{b(x_i)}f_+^2(x_i)\right)^2=2h^2(f)\left(\sum_{x\in V}{b(x)}f_+^2(x)\right)^2,
\end{align*}
where we have used several times that $f_+(x_k)=0$ for any $k\le n_0$.
Applying the last estimate to \eqref{W2} we obtain 
$$
W \succeq \dfrac{2 h^2(f)}{(4-2W)},
$$
from where immediately follows the statement of the lemma.
\end{proof}

\begin{proof}[Proof of Theorem \ref{thmlambda}]
The proof follows exactly the same outline as proof of the Theorem 3.2 in \cite{BauerJost} and we do not repeat it here. 
Note that the facts that 
$$
\lambda_{N-1}=\sup_{f\ne 0}\dfrac{\langle \L f,f\rangle}{\langle f,f\rangle}=\max_{f\ne 0}\dfrac{\langle \L f,f\rangle}{\langle f,f\rangle},
$$ 
where maximum is attained on the corresponding eigenfunction and
$$
2-\lambda_{N-1}=\inf_{f\ne 0}\dfrac{\langle (2\mathcal I-\L) f,f\rangle}{\langle f,f\rangle}, 
$$ 
where $\mathcal I:\F_V\rightarrow \F_V$ is the identical operator, follow from decomposition of $f$ by orthonormal basis, consisting of eigenfunctions of $\L$ (for more details see \cite{MuranovaPreprint22}).

\end{proof}

\section{Examples}\label{LC}
\subsection{Levi-Civita field}
The Levi-Civita field was  introduced by Tullio Levi-Civita in 1862 \cite{LeviCivita}.
The Levi-Civita field is the smallest non-Archimedean real-closed ordered field, which is Cauchy complete in the order topology. Moreover, it has a subfield, isomorphic to the field of rational functions, whose elements naturally appear as edge weights in the theory of electrical networks (see e.g. \cite{Muranova1, Muranova2}).

Firstly, we recall the definition and main notations for the Levi-Civita field $\mathcal R$ (see e.g. \cite{Berz, Hall, Shamseddine, Shamseddinethesis, ShamseddineBerz}).

\begin{definition}
A subset $M$ of the rational numbers $\Bbb Q$ is called \emph{left-finite} if for every number $r\in \Bbb Q$ there are only finitely many elements of $M$ that are smaller than $r$. 
\end{definition}

\begin{definition}\cite{LeviCivita}\label{LCfield}
We define the Levi-Civita field $\mathcal R$ as a field of formal power series
\begin{equation}
a=\sum_{i=0}^\infty a_i\epsilon^{q_i},
\end{equation}
with $a_i\in \Bbb R$ and $Q=\{q_i\}_{i=0}^\infty\subset \mathbb Q$ being left-finite.
\end{definition}
Addition and multiplication are defined naturally as for formal power series. The order is defined as follows: $a\succ 0$ if $a_0>0$.
 
Note that $\epsilon$ here is considered as a `fixed' variable. It is infinitesimal, as well as e.g. $\epsilon^2$, $5\epsilon^\frac{1}{2}+2\epsilon$, etc.

From the definition of elements of the Levi-Civita field follows that to prove the convergence $r_n\to r\in \R$ it is enough to prove, that for any $m\in \Bbb N$ exists $N_0\in \Bbb N$ such that
$|r_n- r|<\epsilon^m$ for any $n>N_0$.
\subsection{Examples of graphs over $\mathcal R$}
\begin{example}
Let us consider a family of complete graphs $V(n)$ with $3$ vertices and weights as at the Figure \ref{fig:example} (i.e. $b(x,y)=\epsilon^n$ in $V(n)$).

\begin{figure}[H]
\centering

\begin{tikzpicture}[scale=0.75]

    \draw [line width=1.5pt, fill=gray!2] (0,0) -- (60:4) -- (4,0) -- cycle;

    \coordinate[label=left:$x$]  (x) at (0,0);
    \coordinate[label=right:$y$] (y) at (4,0);
    \coordinate[label=above:$z$] (z) at (2,3.464);

    \coordinate[label=below:$\epsilon^n$](c) at ($ (x)!.5!(y) $);
    \coordinate[label=left:$1$] (b) at ($ (x)!.5!(z) $);
    \coordinate[label=right:$1$](a) at ($ (y)!.5!(z) $);

\end{tikzpicture}
%
%
%
%
%
%
%
\caption{}
\label{fig:example}
\end{figure}
 One can check, that non-constant eigenvectors of the Laplacian are 

$$
v_1(n)=(-1,1,0) \mbox { and }
v_2(n)=\left(-\dfrac{1}{1+\epsilon^n},-\dfrac{1}{1+\epsilon^n},1\right)
$$
with the corresponding eigenvalues 
$$\lambda_1(n)=\dfrac{1+2\epsilon^n}{1+\epsilon^n} \mbox{ and }
\lambda_2(n)=\dfrac{2+\epsilon^n}{1+\epsilon^n}=2-\epsilon^n+\epsilon^{2n}-\epsilon^{3n}\dots,
$$
where for $\lambda_2(n)$ we have represented ratio as formal power series.

Considering all the posibilities, one can see that the dual Cheeger constant
$$
\overline h(n)=\dfrac{2}{2+\epsilon^n}=1-\dfrac{\epsilon^n}{2}+\dfrac{\epsilon^{2n}}{4}-\dfrac{\epsilon^{3n}}{8}\dots
$$
is attained on the partition of vertices $\{x,y\}\cup\{z\}$.
Then 
$$
1+\sqrt{1-(1-\overline h(n))^2}=2-\dfrac{\epsilon^{2n}}{8}+\dfrac{\epsilon^{3n}}{8}\dots.
$$
Note that the difference 
$$
1+\sqrt{1-(1-\overline h(n))^2}-2\overline h(n)\to 0, \mbox{ as }n\to\infty
$$
in a sence of the ordered topology. Indeed, the graph $V(n)$ differ from a bipartite graph by the edge $x\sim y$, whose weight $b(x,y)=\epsilon^n$  is infinitesimal and the weights converge to $0$ as $n\to\infty$.
\end{example}
\begin{example}
Let us consider a complete graph $V$ with $N=2K$ vertices and the following edge weights: let fix a partition of vertices $V=V_1\cup V_2$ with $\#V_1=\#V_2=K$ and let $b(x,y)=\epsilon^n$ (where  $n\in \Bbb N$ is fixed) if $x,y$ belong to the same $V_i$ and $b(x,y)=1$ otherwise. Firstly we calculate the dual Cheeger constant for this graph. For any $x\in V$ we have $b(x)={K+(K-1)\epsilon^n}$. Let $U_1, U_2$ be disjoint subsets of $V$.

\textbf{1 case.} Let $\#U_1=K_1> K$ and then $\#U_2=K_2< K$. In this case
$$
\overline h(U_1,U_2):=\dfrac{2\sum_{x\in U_1,y\in U_2}b(x,y)}{(K_1+K_2)(K+(K-1)\epsilon^n)}
$$
is maximazed for fixed $K_1, K_2$ as
$$
\overline h(K_1,K_2):=\dfrac{K K_2+(K_1-K_2)\epsilon^n}{(K_1+K_2)(K+(K-1)\epsilon^n)},
$$
i.e. when $U_1\supset V_1$ (or $U_1\supset V_2$). Further,
\begin{align}
\overline h(K_1,K_2)=&\dfrac{2(K K_2+(K_1-K)\epsilon^n)}{(K_1+K_2)(K+(K-1)\epsilon^n)}\\
\preceq& \dfrac{2K (K_2+1)}{(K+1+K_2)(K+(K-1)\epsilon^n)},
\end{align}
since $K_1\succeq K+1$ and $(K_1-K)\epsilon^n$ is infinitesimal. Therefore,
$$
\overline h(K_1,K_2)\le \overline h(V_1,U),
$$
where $\#U=K_2+1, U\subset V_2$ and it is enough to consider the following case.

\textbf{2 case.} Let $\#U_1=K_1\leq K$ and $\#U_2=K_2\leq K$. If
$$
\overline h(U_1,U_2):=\dfrac{2\sum_{x\in U_1,y\in U_2}b(x,y)}{(K_1+K_2)(K+(K-1)\epsilon^n)}
$$
then the maximum of $\overline h(U_1,U_2)$ for fixed $K_1,K_2$ is attained when all the edges between $U_1$ and $U_2$ have weight $1$ (which is possible for $U_1\subset V_1$ and $U_2\subset V_2$). This maximum is
$$
\overline h(K_1,K_2)=\dfrac{2K_1 K_2}{(K_1+K_2)(K+(K-1)\epsilon^n)}
$$
and further, maximazing the ratio ${K_1 K_2}/(K_1+K_2)$ gives
$$
\max_{K_1\leq K, K_2\leq K}\overline h(K_1,K_2)=\dfrac{2 K^2}{2K(K+(K-1)\epsilon^n)}=\dfrac{K}{K+(K-1)\epsilon^n},
$$
where the maximum is attained on $K_1=K_2=K$. 

From all the above, the dual Cheeger constant for the given graph is
\begin{align*}
\overline h&=\dfrac{K}{K+(K-1)\epsilon^n}\\
&=1-\left(\frac{K-1}{K}\right)\epsilon^n+\left(\frac{K-1}{K}\right)^2\epsilon^{2n}+\left(\frac{K-1}{K}\right)^3\epsilon^{3n} \dots,
\end{align*}
where we have represented ratio as formal power series.

Therefore, taking $n$ large enough we can get $|1-\overline h|\prec r$ for any given $r\in \mathcal R$. It is an intuitively clear result, since such a graph is close to bipartite (all edges in the same part of partition have weight $\epsilon^n$, while all edges between $V_1$ and $V_2$ have weight $1$).

The matrix of Laplacian in this case is 
$$
\frac{1}{B}
 \left(
    \begin{array}{ccccc|ccccc}
      B&-\epsilon^n&\dots&-\epsilon^n&-\epsilon^n& -1&-1&\dots&-1&-1\\
      -\epsilon^n&B&\dots&-\epsilon^n&-\epsilon^n& -1&-1&\dots&-1&-1\\
	\dots&\dots&\dots&\dots&\dots&\dots&\dots&\dots&\dots&\dots\\
    -\epsilon^n&-\epsilon^n&\dots&B&-\epsilon^n& -1&-1&\dots&-1&-1\\
   -\epsilon^n&-\epsilon^n&\dots&-\epsilon^n&B& -1&-1&\dots&-1&-1\\
      \hline
      -1&-1&\dots&-1&-1&B&-\epsilon^n&\dots&-\epsilon^n&-\epsilon^n\\
      -1&-1&\dots&-1&-1&-\epsilon^n&B&\dots&-\epsilon^n&-\epsilon^n\\
	\dots&\dots&\dots&\dots&\dots&\dots&\dots&\dots&\dots&\dots\\
      -1&-1&\dots&-1&-1&-\epsilon^n&-\epsilon^n&\dots&B&-\epsilon^n\\
      \multicolumn{5}{c|}{\underbrace{-1\;\;\;\;-1\;\;\;\dots\;\;\;-1\;\;-1}_{K}}&\multicolumn{5}{c}{\underbrace{-\epsilon^n\;\;-\epsilon^n\;\;\dots\;\;\;-\epsilon^n\;\;\;\;B}_{K}}\\
    \end{array}
    \right),
$$
where $B=b(x)={K+(K-1)\epsilon^n}$. Then the eigenvectors are:
\begin{itemize}
\item
$
v_0=(1,1,\dots,1,1)
$
with an eigenvalue $\lambda_0=0$;
\item
$
v_m=(-1,0,\dots, 0,1, 0,\dots, 0),$ where $1$ stays at the place $(m+~1)$ for $m=1,\dots,K-1$ with eigenvalues $$\lambda_m=\dfrac{K+K\epsilon^n}{K+(K-1)\epsilon^n};
$$
\item
$
v_m=(\underbrace{0,0,0\dots 0}_K,-1,0,\dots, 0,1, 0,\dots,0), 
$
where $1$ stays at the place $(m+2)$ for $m=K, \dots, N-2$
with the same eigenvalues as in the previous case
$$\lambda_m=\dfrac{K+K\epsilon^n}{K+(K-1)\epsilon^n};
$$
\item$
v_{N-1}=({\underbrace{1,1,\dots,1,1}_{K}}, {\underbrace{-1,-1,\dots,-1,-1}_{K}}),
$
with the eigenvalue $$\lambda_{N-1}=\dfrac{2K}{K+(K-1)\epsilon^n}.
$$
\end{itemize}
Note that $\lambda_{N-1}=2\overline h$ and can be written as formal power series as
\begin{align*}
\lambda_{N-1}=2-2\frac{K-1}{K}\epsilon^n+2\left(\frac{K-1}{K}\right)^2\epsilon^{2n}-2\left(\frac{K-1}{K}\right)^3\epsilon^{3n} \dots.
\end{align*}
From the other hand,
\begin{align*}
 1+\sqrt{1-(1-\overline h)^2}&=1+\sqrt{\dfrac{K^2+2K(K-1)\epsilon^n}{(K+(K-1)\epsilon^n)^2}}\\
\\
&=2-\frac{(K-1)^2}{2 K^2}\epsilon^{2n}+\frac{(K-1)^3}{K^3}\epsilon^{3n} \dots.
\end{align*}
Therefore, the estimate
$$
\lambda_{N-1}\preceq  1+\sqrt{1-(1-\overline h)^2}
$$
is precise up to the order $\epsilon^n$ for the considered graph and the difference
$$
1+\sqrt{1-(1-\overline h)^2} - \lambda_{N-1}
$$
converge to $0$ in a sence of convergence in order topology in $\R$ if we consider a family of graphs with $n\to\infty$. 

\end{example}

\end{document}